\documentclass{amsart}

\usepackage{tikz}
\usepackage{graphicx}
\usepackage{fullpage}
\usepackage{titlesec}
\usepackage{amsmath}
\usepackage{amssymb}
\usepackage{amsthm}
\usepackage{hyperref}
\usepackage{enumitem}

\newtheorem{theorem}{Theorem}[section]
\newtheorem{proposition}[theorem]{Proposition}
\newtheorem{conjecture}[theorem]{Conjecture}
\newtheorem{corollary}[theorem]{Corollary}
\newtheorem{lemma}[theorem]{Lemma}
\newtheorem*{claim*}{Claim}

\theoremstyle{definition}

\newtheorem{remark}[theorem]{Remark}
\newtheorem{example}[theorem]{Example}

\numberwithin{equation}{section}

\titleformat {name=\section}
  {\Large \bfseries}{}{0px}{\thesection\hspace{10px}}
\titleformat{name=\section,numberless}{\Large \bfseries}{}{0pt}{}

\author{Andrew Weinfeld}

\begin{document}

\title{Bases for Quotients of Symmetric Polynomials}

\maketitle

\begin{abstract}
  We create several families of bases for the symmetric polynomials.  From these bases we prove that certain Schur symmetric polynomials form a basis for quotients of symmetric polynomials that generalize the cohomology and the quantum cohomology of the Grassmannian.  Our work also provides an alternative proof of a result due to Grinberg.
\end{abstract}

\section{Introduction}

In \cite{basisquot} Grinberg studied the quotient ring
\begin{align*}
\mathcal{S}/\langle h_{n-k+1}-a_1,\dots,h_n-a_k \rangle
= \mathcal{S}/\mathcal{I}
\end{align*}
where $\mathcal{S}$ is the ring of symmetric polynomials in $k$ variables over some commutative ring $\mathcal{R}$, $h_{n-k+i}$ is the $(n-k+i)^{th}$ complete homogeneous symmetric polynomial in $\mathcal{S}$, the $a_i$ are arbitrary members of $\mathcal{S}$ satisfying $\deg(a_i)<n-k+i$, and $\mathcal{I}$ is the ideal of $\mathcal{S}$ given by
$$
\mathcal{I} = \langle h_{n-k+1}-a_1,\dots,h_n-a_k \rangle.
$$

This quotient is of interest because it generalizes the ordinary and quantum cohomologies of the Grassmannian.  When $\mathcal{R}=\mathbb{Z}$ and $a_i=0$ the quotient $\mathcal{S}/\mathcal{I}$ becomes isomorphic to $\text{H}^*\,(Gr(k,n))$, the cohomology of the Grassmannian.  When $\mathcal{R}=\mathbb{Z}[q]$, $a_i=0$ for $i<k$, and $a_k=(-1)^{k+1}q$ the quotient $\mathcal{S}/\mathcal{I}$ becomes isomorphic to $\text{QH}^*(\,Gr(k,n))$, the quantum cohomology of the Grassmannian (see \cite{postnikov}).

Grinberg established (see \cite[Theorem 2.7]{basisquot}) that
\begin{align}
\label{equIntroSBasis}
\{s_\lambda\,&|\,\lambda\in P_{k,n-k}\}
\end{align}
is an $\mathcal{R}$-basis for the quotient $\mathcal{S}/\mathcal{I}$,
where $s_\lambda$ is a Schur polynomial and $P_{k,n-k}$ is the set of partitions with at most $k$ parts, each at most $n-k$.  In his proof, Grinberg first uses the Jacobi-Trudi identity to show that (\ref{equIntroSBasis}) spans $\mathcal{S}/\mathcal{I}$ and then computes an explicit Gr{\"o}bner basis for a related quotient to deduce that $\text{rank}_\mathcal{R}(\mathcal{S}/\mathcal{I})=\#P_{k,n-k}$.  Grinberg also proved (see \cite[Theorem 9.9]{basisquot} and \cite[Theorem 9.11]{basisquot}) that
\begin{align*}
\{m_\lambda\,&|\,\lambda\in P_{k,n-k}\}
\\
\text{and }\{h_\lambda\,&|\,\lambda\in P_{k,n-k}\}
\end{align*}
are $\mathcal{R}$-bases for $\mathcal{S}/\mathcal{I}$, where $m_\lambda$ is a monomial symmetric polynomial.

In this paper we will study the quotient ring
$$
\mathcal{S}/\langle p_{n-k+1}-a_1,\dots,p_n-a_k \rangle
= \mathcal{S}/\mathcal{J}
$$
for $\mathcal{R}$ a commutative $\mathbb{Q}$-algebra, where the $h_{n-k+i}$'s from the definition of $\mathcal{S}/\mathcal{I}$ are replaced with $p_{n-k+i}$'s, the power sum symmetric polynomials, and $\mathcal{J}$ is the ideal of $\mathcal{S}$ given by 
$$
\mathcal{J}=\langle p_{n-k+1}-a_1,\dots,p_n-a_k \rangle.
$$

Grinberg's methods for $\mathcal{S}/\mathcal{I}$ cannot be applied to $\mathcal{S}/\mathcal{J}$ because there is no analogue of the Jacobi-Trudi identity for the $p_{n-k+i}$'s and there is no equivalent of the Gr{\"o}bner basis of \cite[Proposition 4.1]{basisquot} for our ideal $\mathcal{J}$.  Therefore, we will need to approach the power sum symmetric polynomials from a different perspective. In particular, we will construct bases for $\mathcal{S}$ from which bases for $\mathcal{S}/\mathcal{J}$ can be deduced. This method can also be used to provide an alternative of Grinberg's result that (\ref{equIntroSBasis}) is a basis for $\mathcal{S}/\mathcal{I}$, without using other quotients or Gr{\"o}bner bases.

Our main results are Theorem~\ref{thmMain} and Corollary~\ref{corMainSME}.

\begin{theorem}
\label{thmMain}
Let $Q_{n-k+1,n}$ be the set of partitions whose parts are weakly between $n-k+1$ and $n$. Then:

When $\mathcal{R}$ is a commutative $\mathbb{Q}$-algebra, the set
\begin{align}
\label{equSPBasis}
\{s_\lambda p_\mu \,|\, \lambda\in P_{k,n-k}\text{, } \mu\in Q_{n-k+1,n}\}
\end{align}
is an $\mathcal{R}$-basis for $\mathcal{S}$.

When $\mathcal{R}$ is any commutative ring, the set
\begin{align}
\label{equSHBasis}
\{s_\lambda h_\mu \,|\, \lambda\in P_{k,n-k}\text{, } \mu\in Q_{n-k+1,n}\}
\end{align}
is an $\mathcal{R}$-basis for $\mathcal{S}$.

Furthermore, the $s_\lambda$'s in Theorem~\ref{thmMain} may be replaced by $m_\lambda$'s or by $e_{\lambda'}$'s, where $e_\lambda$ is an elementary symmetric polynomial and $\lambda'$ denotes the conjugate of $\lambda$.
\end{theorem}

\begin{corollary}
\label{corMainSME}
The sets
\begin{align*}
\{s_\lambda \,&|\, \lambda\in P_{k,n-k} \},
\\
\{m_\lambda \,&|\, \lambda\in P_{k,n-k} \},
\\
\textnormal{and } \{e_{\lambda'} \,&|\, \lambda\in P_{k,n-k} \}
\end{align*}
are $\mathcal{R}$-bases for $\mathcal{S}/\mathcal{J}$ and for $\mathcal{S}/\mathcal{I}$.
\end{corollary}

The rest of this paper is structured as follows.  In Section 2 we review the theory of symmetric polynomials.  In Section 3 we will prove Theorem~\ref{thmMain} for the $p_i$'s, and in Section 4 we will prove Corollary~\ref{corMainSME} for $\mathcal{S}/\mathcal{J}$.  In Section 5 we will prove Theorem~\ref{thmMain} for the $h_i$'s, and in Section 6 we will prove Corollary~\ref{corMainSME} for $\mathcal{S}/\mathcal{I}$.  In Section 7 we will state a couple of conjectures.

\section{Preliminaries}

In this section we will review the theory of symmetric polynomials. The reader may wish to consult \cite[Section 7]{enum2} or \cite{macdonald} for a more detailed treatment.

Let $\mathcal{R}$ be a commutative ring. Fix some positive integer $k$.  We set $\mathcal{S}=\mathcal{R}[x_1,\dots,x_k]^{\mathfrak{S}_k}$ to be the set of elements of $\mathcal{R}[x_1,\dots,x_k]$ that are invariant under all permutations of $x_1,\dots,x_k$, that is, $\mathcal{S}$ is the set of symmetric polynomials in $\mathcal{R}[x_1,\dots,x_k]$.  Then $\mathcal{S}$ is a commutative $\mathcal{R}$-algebra.
 
A partition $\lambda=(\lambda_1,\lambda_2,\dots,\lambda_{\ell(\lambda)})$ is a weakly decreasing sequence of positive integers of length $\ell(\lambda)$.  We set $\lambda_i=0$ for $i>\ell(\lambda)$. We say that $|\lambda|=\sum_{i\geq 1} \lambda_i$.  The conjugate of a partition $\lambda'$ is the partition with parts $\lambda'_i = \#\{\lambda_j\geq i\}$.  Note that $|\lambda'|=|\lambda|$.  Set:
\begin{align*}
P &= \{\lambda \textnormal{ a partition}\}
\\
P(i) &= \{\lambda \in P \textnormal{ and } |\lambda|=i\}
\\
P_k &= \{\lambda \in P \textnormal{ and } \ell(\lambda)\leq k\}
\\
P_k(i) &= \{\lambda \in P_k\textnormal{ and }|\lambda|=i\}
\\
P_{k,n-k} &= \{\lambda\in P_k \textnormal{ and }\lambda_1\leq n-k\}
\\
P_{k,n-k}(i) &= \{\lambda \in P_{k,n-k}\textnormal{ and }|\lambda|=i\}
\\
Q_{n-k+1,n} &= \{\lambda\in P \textnormal{ and }n\geq \lambda_1\geq\cdots\geq\lambda_{\ell(\lambda)}\geq n-k+1\}
\\
Q_{n-k+1,n}(i) &= \{\lambda \in Q_{n-k+1,n}\textnormal{ and }|\lambda|=i\}
\end{align*}
Note that $\lambda'\in P_k\iff \lambda_1\leq k$, and that members of $Q_{n-k+1,n}$ could have any length.

The majorization ordering $\lambda \unlhd \mu$ is a poset structure on $P$ and is given by
$$
\lambda \unlhd \mu \iff \forall j: \lambda_1+\cdots+\lambda_j \leq \mu_1+\cdots+\mu_j.
$$

Recall the following well-known families of symmetric polynomials.

\begin{itemize}
    \item The monomial symmetric polynomials,
    $$m_\lambda = \sum_{\substack{a_1,\dots,a_k\text{ a distinct} \\ \text{permutation of }\lambda_1,\dots,\lambda_k}} x_1^{a_1}\cdots x_k^{a_k}$$
    when $\ell(\lambda)\leq k$. We set $m_\emptyset=1$ and $m_\lambda=0$ when $\ell(\lambda)>k$.
    \item The elementary symmetric polynomials,
    \begin{align*}e_i&=\sum_{1\leq j_1 <\cdots < j_i \leq k} x_{j_1}\cdots x_{j_i}
    \\ \text{and } e_\lambda &= e_{\lambda_1}\cdots e_{\lambda_{\ell(\lambda)}}.\end{align*}
    We set $e_0=e_\emptyset=1$.  Note that $e_i=0$ when $i>k$.
    \item The complete homogeneous symmetric polynomials,
    \begin{align*}h_i&=\sum_{1\leq j_1 \leq \cdots \leq j_i \leq k} x_{j_1}\cdots x_{j_i}
    \\ \text{and } h_\lambda &= h_{\lambda_1}\cdots h_{\lambda_{\ell(\lambda)}}.\end{align*}
    We set $h_0=h_\emptyset=1$.
    \item The Schur symmetric polynomials, $$s_\lambda = \frac{\det(x_j^{\lambda_i+k-i})_{i,j=1}^k}{\det(x_j^{k-i})_{i,j=1}^k}$$
    when $\ell(\lambda)\leq k$.  We set $s_\emptyset=1$ and $s_\lambda=0$ when $\ell(\lambda)>k$.
    \item The power sum symmetric polynomials,
    \begin{align*}p_i&= \sum_{j=1}^k x_j^i
    \\ \text{and } p_\lambda &= p_{\lambda_1}\cdots p_{\lambda_{\ell(\lambda)}}\end{align*}
    for $i\geq1$.  We set $p_\emptyset=1$.
\end{itemize}
Each of the above families contains a basis for $\mathcal{S}$, specifically:
\begin{itemize}
    \item $m_\lambda$ for $\lambda\in P_k$
    \item $e_\lambda$ for $\lambda'\in P_k$
    \item $h_\lambda$ for $\lambda'\in P_k$
    \item $s_\lambda$ for $\lambda\in P_k$
    \item When $\mathcal{R}$ is a commutative $\mathbb{Q}$-algebra: $p_\lambda$ for $\lambda'\in P_k$
\end{itemize}
Two well-known recursive identities state that for $i>k$, we have
\begin{align}
\label{equNewtonGirardpe}
p_i &= \sum_{j=1}^k (-1)^{j+1}e_j p_{i-j} \\
\label{equNewtonGirardhe}
\text{and } h_i &= \sum_{j=1}^k (-1)^{j+1}e_j h_{i-j}.
\end{align}
These follow from \cite[(2.11')]{macdonald} and \cite[(7.13)]{enum2} respectively upon setting $x_{k+1},x_{k+2},\dots$ to $0$, noting that this sets $e_{k+1},e_{k+2},\dots$ to $0$, recalling that $e_0=1$, and rearranging.  The first identity (\ref{equNewtonGirardpe}) is one of the Newton-Girard identities.

When $\ell(\lambda)\leq k$, the Jacobi-Trudi identity states that
\begin{align}
\label{equJacobiTrudi}
s_\lambda = \det(h_{\lambda_i+j-i})_{i,j=1}^{k}
\end{align}
where $h_i=0$ for $i<0$ (see \cite[(3.4)]{macdonald}).

We set $\mathcal{S}_i=\textnormal{span}_\mathcal{R}(m_\lambda\,|\,\lambda\in P_k(i))$ to be the set of all homogeneous polynomials of degree $i$ in $\mathcal{S}$.  Then $\mathcal{S}$ is a graded ring
$$
\mathcal{S} = \mathcal{S}_0\oplus \mathcal{S}_1\oplus\cdots
$$
and each of the bases for $\mathcal{S}$ becomes a basis for $\mathcal{S}_i$ upon restricting to those $\lambda$ that satisfy $|\lambda|=i$.  In particular, $\mathcal{S}_i$ is a free and finite $\mathcal{R}$-module of rank $\#P_k(i)$.

We will use the terminology of \cite[Remark 11.1.17]{expansions} and \cite[Corollary 11.1.19]{expansions} to discuss the transition matrices between the bases of $\mathcal{S}_i$. In particular, we will need the following:
\begin{itemize}
    \item The family $\{s_\nu\}_{\nu\in P_k(i)}$ expands unitriangularly in $\{m_\nu\}_{\nu\in P_k(i)}$ under $\unlhd$, that is,
    \begin{align}
        \label{equExpandSinM}
        s_\lambda = \sum_{\substack{\mu\in P_k(i)\\\mu \unlhd \lambda}} M(s,m)_{\lambda,\mu} m_{\mu}
    \end{align}
    where $M(s,m)_{\lambda,\mu}\in\mathcal{R}$ and $M(s,m)_{\lambda,\lambda}=1$.
    \item The family $\{e_{\nu'}\}_{\nu\in P_k(i)}$ expands unitriangularly in $\{s_\nu\}_{\nu\in P_k(i)}$ under $\unlhd$, that is,
    \begin{align}
        \label{equExpandEinS}
        e_{\lambda'} = \sum_{\substack{\mu\in P_k(i)\\\mu \unlhd \lambda}} M(e,s)_{\lambda,\mu} s_{\mu}
    \end{align}
    where $M(e,s)_{\lambda,\mu}\in\mathcal{R}$ and $M(e,s)_{\lambda,\lambda}=1$.
\end{itemize}
The expansions (\ref{equExpandSinM}) and (\ref{equExpandEinS}) follow from \cite[Table 1]{macdonald} upon setting $x_{k+1},x_{k+2},\dots$ to $0$.

We define $\Lambda=\mathcal{R}[x_1,x_2,\dots]^{\mathfrak{S}_\infty}$, the ring of symmetric functions, to be the ring of power series over $\mathcal{R}$ of bounded degree in countably many variables which are invariant under any permutation of $x_1,x_2,\dots$.  We can define $m_\lambda$ and $p_\lambda$ in $\Lambda$ identically as in $\mathcal{S}$ by setting $k=\infty$ (we will assume that we are working in $\mathcal{S}$ unless otherwise specified).  We set $\Lambda_i$ to be the set of all homogeneous power series of degree $i$ in $\Lambda$ so that $\Lambda$ is a graded ring
$$
\Lambda = \Lambda_0\oplus \Lambda_1\oplus\cdots.
$$
Setting $x_{k+1},x_{k+2},\dots$ to $0$ in a symmetric function gives a symmetric polynomial in $x_1,\dots,x_k$, and setting $x_{k+1},x_{k+2},\dots$ to $0$ is an algebra homomorphism from $\Lambda$ onto $\mathcal{S}$ that sends $m_\lambda$ to $m_\lambda$ and $p_\lambda$ to $p_\lambda$.

\section{Bases with Power Sum Symmetric Polynomials}

For this section let $\mathcal{R}$ be a commutative $\mathbb{Q}$-algebra.  Fix some positive integer $k$.  Set $\mathcal{S}=\mathcal{R}[x_1,\dots,x_k]^{\mathfrak{S}_k}$ to be the ring of symmetric polynomials in $\mathcal{R}[x_1,\dots,x_k]$.

Fix some integer $n\geq k$.

\begin{proposition}
\label{propMixedBasisMlambdaPmu}
The set
$$
\{m_\lambda p_\mu \,|\, \lambda\in P_{k,n-k}\text{, } \mu\in Q_{n-k+1,n}\}
$$
is an $\mathcal{R}$-basis for $\mathcal{S}$.
\end{proposition}

For this section, set
$$
V_i = \{m_\lambda p_\mu \,|\, \lambda\in P_{k,n-k}\text{, } \mu\in Q_{n-k+1,n} \text{, } |\lambda|+|\mu| = i\}
$$
This is the restriction of the basis from Proposition~\ref{propMixedBasisMlambdaPmu} to $\mathcal{S}_i$.

\begin{remark}
\label{rmkStrategyForMixedBasisMlambdaPmu}
Since $m_\lambda p_\mu$ is homogeneous of degree $|\lambda|+|\mu|$, Proposition~\ref{propMixedBasisMlambdaPmu} is equivalent to showing that for $i\geq0$, $V_i$ is an $\mathcal{R}$-basis for $\mathcal{S}_i$.  We will do so by showing that $\#V_i=\textnormal{rank}_\mathcal{R}(\mathcal{S}_i)$ and $\textnormal{span}_\mathcal{R}(V_i)=\mathcal{S}_i$ (that this will imply Proposition~\ref{propMixedBasisMlambdaPmu} is well-known, see \cite[Exercise 2.5.18b]{expansions}\footnote{
Exercise solution: Consider the $\textbf{k}$-linear map $f:A\to A$ given by $f(\gamma_i)=\beta_i$. Then $f$ is a surjective endomorphism, so by \cite[Proposition 1.2]{rankandspanimplybasis}, $f$ is an isomorphism.  Therefore, $(f(\gamma_i))_{i\in I}=(\beta_i)_{i\in I}$ is a basis of $A$.
}).
\end{remark}

\begin{lemma}
\label{lemDimMixedBasisMlambdaPmuWorks}
For $i\geq0$, we have $\#V_i=\textnormal{rank}_\mathcal{R}(\mathcal{S}_i)$.
\end{lemma}
\begin{proof}
Since $P_k(i)$ indexes the $m_\lambda$ basis for $\mathcal{S}_i$, by \cite[(1.76)]{enum1} we have
\begin{align*}
\sum_{i\geq 0} q^i \textnormal{rank}_\mathcal{R}(\mathcal{S}_i)
= \sum_{i\geq 0} q^i \#P_k(i)
= \prod_{j=1}^k \frac{1}{1-q^j}.
\end{align*}
Clearly
\begin{align*}
\sum_{i\geq 0} q^i \#V_i
&= \left(\sum_{i\geq 0} q^i \#P_{k,n-k}(i) \right)
\left(\sum_{i\geq 0} q^i \#Q_{n-k+1,n}(i) \right).
\end{align*}
From \cite[Section 1.6, (6)]{aigner} and \cite[Section 1.6, (4)]{aigner}, we have that
\begin{align*}
\sum_{i\geq 0} q^i \#P_{k,n-k}(i) = \begin{bmatrix} n\\k \end{bmatrix}_q =\prod_{j=1}^k \frac{1-q^{n-k+j}}{1-q^j}.
\end{align*}
From \cite[Corollary 1.8.2]{enum1}, we have that
\begin{align*}
\sum_{i\geq 0} q^i \#Q_{n-k+1,n}(i) = \prod_{j=1}^k \frac{1}{1-q^{n-k+j}}.
\end{align*}
Therefore,
\begin{align}
\label{equDimCounting}
\sum_{i\geq 0} q^i \#V_i
= \left(\prod_{j=1}^k \frac{1-q^{n-k+j}}{1-q^j}\right) \left(\prod_{j=1}^k \frac{1}{1-q^{n-k+j}}\right)
= \prod_{j=1}^k \frac{1}{1-q^j}
= \sum_{i\geq 0} q^i \textnormal{rank}_\mathcal{R}(\mathcal{S}_i).
\end{align}
The result follows from equating coefficients of $q^i$ on both sides of (\ref{equDimCounting}).
\end{proof}

\begin{remark}
One interesting question raised by Lemma~\ref{lemDimMixedBasisMlambdaPmuWorks} is whether a bijective proof that
$$
\#P_k(i)=\#V_i
$$
could be found.  No such proof is known to the author.
\end{remark}

In order to show that $V_i$ spans $\mathcal{S}_i$, we will show that $\textnormal{span}_\mathcal{R}(V_i)$ contains every member of the monomial basis of $\mathcal{S}_i$.  We will now provide an algorithm to convert any $m_\lambda$ (with $\lambda\in P_k(i)$) into an $\mathcal{R}$-linear combination over $V_i$ by induction on $i$.  The steps of the algorithm are:
\begin{enumerate}[label=(\arabic*)]
    \item If $\lambda_1\leq n-k$ then $m_\lambda=m_\lambda p_\emptyset \in V_i$ and we are done. If not, we have $\lambda_1\geq n-k+1$.
    \item Expand $m_\lambda$ in the power sum basis in $\Lambda_i$.  Note that each $p_\mu$ with a nonzero coefficient satisfies $\mu_1\geq\lambda_1\geq n-k+1$.  Set $x_{k+1},x_{k+2},\dots$ to $0$ to project this expansion into $\mathcal{S}_i$.
    \item Repeatedly apply the Newton-Girard identity (\ref{equNewtonGirardpe})
    $$
    p_j = \sum_{t=1}^k (-1)^{t+1}e_t p_{j-t}
    $$
    to convert each $p_{\mu_1}$ into an $\mathcal{S}$-linear combination of $p_{n-k+1},\dots,p_n$.  Note that since all of these computations can be seen as taking place in $\mathcal{S}_{\mu_1}$, we have that the degree of the coefficient of $p_{n-k+t}$ is $\mu_1-(n-k+t)$.
    \item Factor out $p_{n-k+1},\dots,p_n$ from each term and recollect to write $m_\lambda$ as an $\mathcal{S}$-linear combination of $p_{n-k+1},\dots,p_n$ in which the coefficient of $p_m$ has degree $i-m$ for each $m\in \{n-k+1,\dots,n\}$.
    \item By induction, the coefficient of each of $p_{n-k+1},\dots,p_n$ can be written in the basis of Proposition~\ref{propMixedBasisMlambdaPmu}.  Specifically, we expand each coefficient in the $m_\lambda$ basis of $\mathcal{S}$ and repeat the algorithm.
    \item Since multiplying an element of $V_{i-m}$ by $p_m$ with $n-k+1\leq m \leq n$ will result in a member of $V_i$, we see that we have written $m_\lambda$ as an $\mathcal{R}$-linear combination of elements of our basis.
\end{enumerate}

\begin{example}
Let $k=3, n=4$, $\lambda=(2,2,1)$, and consider $m_\lambda=m_{2,2,1}$.
\begin{enumerate}[label=(\arabic*)]
    \item We have $\lambda_1=2\not<1=n-k$, so we proceed to step (2).
    \item In $\Lambda$ we have
    $$
    m_{2,2,1} = \frac{1}{2}p_{2,2,1} - p_{3,2} - \frac{1}{2}p_{4,1} + p_{5}
    $$
    and setting $x_4,x_5,\dots$ to $0$ we see that this equality holds in $\mathcal{S}$ as well.
    \item The only $p_\mu$ with $\mu_1>n$ is $p_{5}$, so we use a Newton-Girard identity to rewrite
    $$
    m_{2,2,1} = \frac{1}{2}p_{2,2,1} - p_{3,2} - \frac{1}{2}p_{4,1} + (e_1 p_4 - e_2 p_3 + e_3 p_2)
    $$
    \item Now we factor out $p_{n-k+1},\dots,p_n=p_2,p_3,p_4$:
    \begin{align}
    \label{equExExp}
    m_{2,2,1} = \left(\frac{1}{2}p_{2,1}+e_3\right)p_2 + (-p_2-e_2)p_3 + \left(- \frac{1}{2}p_1 + e_1\right)p_4
    \end{align}
    \item To convert the coefficients of $p_2$, $p_3$, and $p_4$ into our basis, we convert them into the monomial basis and apply the algorithm, giving:
    \begin{equation}
    \begin{aligned}
    \label{equExIter}
        \frac{1}{2}p_{2,1}+e_3 &= \frac{1}{2}m_1p_2 + m_{1,1,1}
        \\
        -p_2-e_2 &=  -p_2 - m_{1,1}
        \\
        - \frac{1}{2}p_1 + e_1 &=  \frac{1}{2}m_1
    \end{aligned}
    \end{equation}
    \item Substituting (\ref{equExIter}) into (\ref{equExExp}), we see that
    \begin{align*}
        m_{2,2,1} &= \left(\frac{1}{2}m_1p_2 + m_{1,1,1}\right)p_2 + (-p_2 - m_{1,1})p_3 + \left(\frac{1}{2}m_1\right)p_4
        \\
        &= \frac{1}{2}m_1p_{2,2} + m_{1,1,1}p_2 -p_{3,2} - m_{1,1}p_3 + \frac{1}{2}m_1p_4.
    \end{align*}
\end{enumerate}
\end{example}

Here is a formal proof that the algorithm works and terminates.  This also proves Proposition~\ref{propMixedBasisMlambdaPmu}. First, we will show that step (3) works, then we will show that the entire algorithm is valid.

\begin{lemma}
\label{lemReducingPViaNewtonGirard}
Let $t\geq n-k+1$.  Then for some $d_{t,m}\in\mathcal{S}_{t-m}$ (where $n-k+1\leq m\leq n$) we have
$$
p_t = \sum_{m=n-k+1}^n d_{t,m} p_m,
$$
that is, step \textnormal{(3)} is valid.
\end{lemma}
\begin{proof}
We will induct on $t$.

Base case $n-k+1\leq t\leq n$: We can just take $d_{t,m}=\delta_{t,m}$ where $\delta$ is the Kronecker delta.

Induction step $t>n$: Suppose that Lemma~\ref{lemReducingPViaNewtonGirard} holds for all smaller $t\geq n-k+1$. From the Newton-Girard identity (\ref{equNewtonGirardpe}) we have
\begin{align*}
p_t &= \sum_{j=1}^k (-1)^{j+1}e_j p_{t-j}
\\
&= \sum_{j=1}^k (-1)^{j+1}e_j \sum_{m=n-k+1}^n d_{t-j,m} p_m
\\
&= \sum_{m=n-k+1}^n \left( \sum_{j=1}^k (-1)^{j+1}e_j d_{t-j,m} \right) p_m
\end{align*}
so we can take
$$
d_{t,m} = \sum_{j=1}^k (-1)^{j+1}e_j d_{t-j,m}
$$
which is in $\mathcal{S}_{j+t-j-m}=\mathcal{S}_{t-m}$ as desired.
\end{proof}

\begin{lemma}
\label{lemMixedBasisMPContainsMlambda}
Let $i\geq 0$ and $\lambda\in P_k(i)$.  Then $m_\lambda\in\textnormal{span}_\mathcal{R} V_i$. In particular, the algorithm works.
\end{lemma}
\begin{proof}
We will induct on $i$.

Base case $i\leq n-k$: We have $\lambda_1\leq i\leq n-k$ so $m_\lambda=m_\lambda p_\emptyset \in V_i$.

Induction step $i>n-k$:  Suppose that Lemma~\ref{lemMixedBasisMPContainsMlambda} holds for all smaller $i$.  Let $\lambda\in P_k(i)$.  If $\lambda_1 \leq n-k$ then $m_\lambda=m_\lambda p_\emptyset \in V_i$ as desired, so let $\lambda_1\geq n-k+1$.

\begin{claim*}
In $\Lambda$, for some $c_t\in \Lambda_{i-t}$, we have
$$
m_\lambda = \sum_{t=\lambda_1}^i c_t p_t.
$$
\end{claim*}
\begin{proof}[Proof of the claim]
The proof of \cite[Proposition~2.2.10]{expansions} shows that in $\Lambda$, the family $\{p_\nu\}_{\nu\in P(i)}$ expands invertibly triangularly in the family $\{m_\nu\}_{\nu\in P(i)}$ under the reverse of the majorization ordering.  Then by \cite[Corollary~11.1.19~(a)]{expansions} we have that the family $\{m_\nu\}_{\nu\in P(i)}$ expands invertibly triangularly in the family $\{p_\nu\}_{\nu\in P(i)}$ also under the reverse of the majorization ordering.  Specifically, we have for some $b_{\lambda,\mu}\in\mathcal{R}$ that\footnote{
A combinatorial interpretation for the $b_{\lambda,\mu}$'s is given in \cite[(11)]{egecioglu}, namely,
$$b_{\lambda,\mu} = (-1)^{\ell(\lambda)-\ell(\mu)}\frac{w(B_{\lambda,\mu})}{z_\mu}$$
where $z_\mu,w(B_{\lambda,\mu})\in\mathbb{N}$ and $w(B_{\lambda,\mu})$ is a weighted sum that runs across all ways to insert bricks of sizes $\lambda$ into the Young diagram of $\mu$.
}
$$
m_\lambda = \sum_{\substack{\mu\in P(i)\\\mu\unrhd\lambda}}
b_{\lambda,\mu} p_\mu.
$$
For each $\mu$ in the sum, we have $\mu \unrhd \lambda \implies \mu_1 \geq \lambda_1$, and also $\mu_1\leq |\mu|=i$, so it follows that
\begin{align*}
m_\lambda &= \sum_{\mu_1 = \lambda_1}^i \sum_{\substack{\mu\in P(i)\\\mu\unrhd\lambda}} b_{\lambda,\mu} p_{\mu_1}p_{\mu_2}\cdots p_{\mu_{\ell(\mu)}}
\\
&= \sum_{\mu_1 = \lambda_1}^i \left(\sum_{\substack{\mu\in P(i)\\\mu\unrhd\lambda}} b_{\lambda,\mu} p_{\mu_2}\cdots p_{\mu_{\ell(\mu)}}\right) p_{\mu_1}.
\end{align*}
Therefore, we can take
$$
c_t = \sum_{\substack{\mu\in P(i)\\\mu\unrhd\lambda\\\mu_1=t}} b_{\lambda,\mu} p_{\mu_2}\cdots p_{\mu_{\ell(\mu)}}
$$
which is in $\Lambda_{\mu_2+\cdots+\mu_{\ell(\mu)}} = \Lambda_{i-t}$.
\end{proof}

Now, setting $x_{k+1},x_{k+2},\dots$ to $0$ we can write in $\mathcal{S}_i$ that
$$
m_\lambda = \sum_{t=\lambda_1}^i c_t p_t
$$
where $c_t\in \mathcal{S}_{i-t}$.  Applying Lemma~\ref{lemReducingPViaNewtonGirard} to each $t$ in the sum (since $t\geq\lambda_1\geq n-k+1$) gives
\begin{align*}
m_\lambda &= \sum_{t=\lambda_1}^i c_t p_t
\\
&= \sum_{t=\lambda_1}^i c_t \sum_{m=n-k+1}^n d_{t,m} p_m
\\
&= \sum_{m=n-k+1}^n \left( \sum_{t=\lambda_1}^i c_t d_{t,m} \right) p_m
\end{align*}
where $d_{t,m}\in\mathcal{S}_{t-m}$.  Note that $c_t d_{t,m}\in\mathcal{S}_{i-t+t-m}=\mathcal{S}_{i-m}$.  Since $i-m<i$, by our induction hypothesis, we can expand $c_t d_{t,m}$ as an $\mathcal{R}$-linear combination of $V_{i-m}$.  Then since multiplication by $p_m$, $n-k+1\leq m\leq n$, converts a member of $V_{i-m}$ into a member of $V_i$, we see that we have expanded $m_\lambda$ as an $\mathcal{R}$-linear combination of elements of $V_i$.  This concludes the induction step.
\end{proof}

\begin{proof}[Proof of Proposition~\ref{propMixedBasisMlambdaPmu}]
Recall Remark~\ref{rmkStrategyForMixedBasisMlambdaPmu}.  We established that $\#V_i=\textnormal{rank}_\mathcal{R}(\mathcal{S}_i)$ in Lemma~\ref{lemDimMixedBasisMlambdaPmuWorks}.  Then, in Lemma~\ref{lemMixedBasisMPContainsMlambda} we showed that if $\lambda\in P_k(i)$ then $m_\lambda\in \textnormal{span}_\mathcal{R}(V_i)$.  It follows that $\textnormal{span}_\mathcal{R}(V_i)=\mathcal{S}_i$, so $V_i$ is an $\mathcal{R}$-basis for $\mathcal{S}_i$ and the result follows.
\end{proof}

We can now switch the $m_\lambda$'s for $s_\lambda$'s to obtain the basis (\ref{equSPBasis}) from Theorem~\ref{thmMain}.

\begin{lemma}
\label{lemSinMTriangularityPkn-k}
The family $\{s_\nu\}_{\nu\in P_{k,n-k}}$ expands invertibly triangularly in $\{m_\nu\}_{\nu\in P_{k,n-k}}$ under the majorization ordering $\unlhd$.  This lemma still holds when $\mathcal{R}$ is any commutative ring\footnote{We will need this case in Section 5.}.
\end{lemma}
\begin{proof}
Let $w$ be the partition consisting of $k$ copies of $n-k$.  Then $\lambda\in P_{k,n-k}\implies\lambda \unlhd w$, and every partition $\lambda$ with at most $k$ parts satisfies $\lambda \unlhd w\implies\lambda_1 \leq n-k\implies\lambda\in P_{k,n-k}$. Therefore, $\lambda\in P_{k,n-k} \iff \lambda \unlhd w$.

Let $\lambda\in P_{k,n-k}$.  Then for each $\mu$ in the sum (\ref{equExpandSinM}) we have $\mu\unlhd\lambda\unlhd w\implies \mu\unlhd w\implies \mu\in P_{k,n-k}$.  The result follows.
\end{proof}

\begin{corollary}
\label{corMixedBasisSlambdaPmu}
The set
$$
\{s_\lambda p_\mu \,|\, \lambda\in P_{k,n-k}\text{, } \mu\in Q_{n-k+1,n}\}
$$
is an $\mathcal{R}$-basis for $\mathcal{S}$.
\end{corollary}
\begin{proof}
By Proposition~\ref{propMixedBasisMlambdaPmu} we see that a basis for the $\textnormal{span}_\mathcal{R}(p_{\mu}\,|\,
\mu\in Q_{n-k+1,k})$-module $\mathcal{S}$ is $\{m_\nu\}_{\nu\in P_{k,n-k}}$.  Therefore, by \cite[Corollary~11.1.19~(e)]{expansions} and Lemma~\ref{lemSinMTriangularityPkn-k}, we see that $\{s_\nu\}_{\nu\in P_{k,n-k}}$ is a basis for the $\textnormal{span}_\mathcal{R}(p_{\mu}\,|\,
\mu\in Q_{n-k+1,k})$-module $\mathcal{S}$ and the result follows.
\end{proof}

Now we will show that the $s_\lambda$'s of Corollary \ref{corMixedBasisSlambdaPmu} may be replaced by $e_{\lambda'}$'s.

\begin{lemma}
\label{lemEinSTriangularityPkn-k}
The family $\{e_{\nu'}\}_{\nu\in P_{k,n-k}}$ expands invertibly triangularly in $\{s_\nu\}_{\nu\in P_{k,n-k}}$ under the majorization ordering $\unlhd$.  This lemma still holds when $\mathcal{R}$ is any commutative ring\footnote{We will need this case in Section 5.}.
\end{lemma}
\begin{proof}
This is identical to Lemma~\ref{lemSinMTriangularityPkn-k} except with (\ref{equExpandEinS}) instead of (\ref{equExpandSinM}).
\end{proof}

\begin{corollary}
\label{corMixedBasisElambdaPmu}
The set
$$
\{e_{\lambda'} p_\mu \,|\, \lambda\in P_{k,n-k}\text{, } \mu\in Q_{n-k+1,n}\}
$$
is an $\mathcal{R}$-basis for $\mathcal{S}$.
\end{corollary}
\begin{proof}
This is identical to Corollary~\ref{corMixedBasisSlambdaPmu} except with $e_{\nu'}$ instead of $s_\nu$, with $s_\nu$ instead of $m_\nu$, and with Lemma~\ref{lemEinSTriangularityPkn-k} instead of Lemma~\ref{lemSinMTriangularityPkn-k}.
\end{proof}

\begin{remark}
\label{rmkHalfofThmMain}
Corollary~\ref{corMixedBasisSlambdaPmu}, Proposition~\ref{propMixedBasisMlambdaPmu}, and Corollary~\ref{corMixedBasisElambdaPmu} prove half of Theorem~\ref{thmMain}.
\end{remark}

\section{Quotients with Power Sum Symmetric Polynomials}

For this section let $\mathcal{R}$ be a commutative $\mathbb{Q}$-algebra.  Recall that in the definition of $\mathcal{J}$ we had $a_i\in\mathcal{S}$ with $\deg(a_{n-k+i})<n-k+i$ for $i=1,\dots,k$.  Set $b_{n-k+i}=a_i$ (this is just a reindexing). In particular, $\deg(b_{n-k+i})<n-k+i$ and
$$
\mathcal{J} = \langle p_{n-k+1}-b_{n-k+1},\dots,p_n-b_n \rangle.
$$

First, we create slightly more general versions of the bases from Corollary~\ref{corMixedBasisSlambdaPmu}, Proposition~\ref{propMixedBasisMlambdaPmu}, and Corollary~\ref{corMixedBasisElambdaPmu}.

\begin{corollary}
\label{corMixedBasesSMElambdaPmuAi}
The sets
\begin{align*}
\left\{\left.s_\lambda \prod_{j=1}^{\ell(\mu)} (p_{\mu_j}-b_{\mu_j}) \,\right|\, \lambda\in P_{k,n-k}\text{, } \mu\in Q_{n-k+1,n}\right\}&,
\\
\left\{\left.m_\lambda \prod_{j=1}^{\ell(\mu)} (p_{\mu_j}-b_{\mu_j}) \,\right|\, \lambda\in P_{k,n-k}\text{, } \mu\in Q_{n-k+1,n}\right\}&,
\\
\text{and }\left\{\left.e_{\lambda'} \prod_{j=1}^{\ell(\mu)} (p_{\mu_j}-b_{\mu_j}) \,\right|\, \lambda\in P_{k,n-k}\text{, } \mu\in Q_{n-k+1,n}\right\}&
\end{align*}
are $\mathcal{R}$-bases for $\mathcal{S}$.
\end{corollary}
\begin{proof}
It suffices to show that the members of the sets, when restricted to those with degree at most $i$, are a basis for $\mathcal{S}_0\oplus\cdots\oplus\mathcal{S}_i$.  In this sense, the first set expands invertibly triangularly in the basis from Corollary~\ref{corMixedBasisSlambdaPmu} under the partial ordering $(\lambda,\mu)<(\rho,\tau)\iff |\lambda|+|\mu|<|\rho|+|\tau|$.  The second set expands invertibly triangularly in the basis from Proposition~\ref{propMixedBasisMlambdaPmu} under the same ordering.  The third set expands invertibly triangularly in the basis from Corollary~\ref{corMixedBasisElambdaPmu} under the same ordering.  Then all three sets are bases by \cite[Corollary~11.1.19~(e)]{expansions}.
\end{proof}

\begin{corollary}
\label{corSMEWeinQuoBasis}
The sets
\begin{align*}
\{s_\lambda\,&|\,\lambda\in P_{k,n-k}\},
\\
\{m_\lambda\,&|\,\lambda\in P_{k,n-k}\},
\\
\text{and } \{e_{\lambda'}\,&|\,\lambda\in P_{k,n-k}\}
\end{align*}
are $\mathcal{R}$-bases for $\mathcal{S}/\mathcal{J}$.
\end{corollary}
\begin{proof}
From Corollary~\ref{corMixedBasesSMElambdaPmuAi} we see that three $\mathcal{R}$-bases for $\mathcal{J}$ are
\begin{align*}
\left\{\left.s_\lambda \prod_{j=1}^{\ell(\mu)} (p_{\mu_j}-b_{\mu_j})\,\right|\, \lambda\in P_{k,n-k},\ \mu \in Q_{n-k+1,n}, \ell(\mu)\geq 1\right\}&,
\\
\left\{\left.m_\lambda \prod_{j=1}^{\ell(\mu)} (p_{\mu_j}-b_{\mu_j})\,\right|\, \lambda\in P_{k,n-k},\ \mu \in Q_{n-k+1,n}, \ell(\mu)\geq 1\right\}&,
\\
\text{and }\left\{\left.e_{\lambda'} \prod_{j=1}^{\ell(\mu)} (p_{\mu_j}-b_{\mu_j})\,\right|\, \lambda\in P_{k,n-k},\ \mu \in Q_{n-k+1,n}, \ell(\mu)\geq 1\right\}&.
\end{align*}
Therefore, by Corollary~\ref{corMixedBasesSMElambdaPmuAi}, three bases for $\mathcal{S}/\mathcal{J}$ are 
\begin{align*}
\{s_\lambda\,&|\,\lambda\in P_{k,n-k}\},
\\
\{m_\lambda\,&|\,\lambda\in P_{k,n-k}\},
\\
\text{and } \{e_{\lambda'}\,&|\,\lambda\in P_{k,n-k}\}.
\end{align*}
\end{proof}

\section{Bases with Complete Homogeneous Symmetric Polynomials}

For this section let $\mathcal{R}$ be any commutative ring.  Fix some positive integer $k$.  We set $\mathcal{S}=\mathcal{R}[x_1,\dots,x_k]^{\mathfrak{S}_k}$ to be the ring of symmetric polynomials in $\mathcal{R}[x_1,\dots,x_k]$.

Fix some integer $n\geq k$.

\begin{remark}
We can do with the complete homogeneous symmetric polynomials almost exactly what we did with the power sum symmetric polynomials with only a few minor changes, such as exchanging the monomial and Schur symmetric polynomials and using the Jacobi-Trudi identity in the analogue of Lemma~\ref{lemMixedBasisMPContainsMlambda}.
\end{remark}

\begin{theorem}
\label{thmMixedBasisSlambdaHmu}
The set
$$
\{s_\lambda h_\mu \,|\, \lambda\in P_{k,n-k}\text{, } \mu\in Q_{n-k+1,n}\}
$$
is an $\mathcal{R}$-basis for $\mathcal{S}$.
\end{theorem}

For this section, set
$$
V_i = \{s_\lambda h_\mu \,|\, \lambda\in P_{k,n-k}\text{, } \mu\in Q_{n-k+1,n} \text{, } |\lambda|+|\mu| = i\}.
$$
This is the restriction of our basis to $\mathcal{S}_i$.

\begin{remark}
\label{rmkStrategyForMixedBasisSlambdaHmu}
Since $s_\lambda h_\mu$ is homogeneous of degree $|\lambda|+|\mu|$, we can prove Theorem~\ref{thmMixedBasisSlambdaHmu} by showing that for $i\geq0$, $V_i$ is an $\mathcal{R}$-basis for $\mathcal{S}_i$.  We will do so by showing that $\#V_i=\textnormal{rank}_\mathcal{R}(\mathcal{S}_i)$ and $\textnormal{span}_\mathcal{R}(V_i)=\mathcal{S}_i$.  Showing that $\#V_i=\textnormal{rank}_\mathcal{R}(\mathcal{S}_i)$ is identical to Lemma~\ref{lemDimMixedBasisMlambdaPmuWorks}, so we only need to establish that $\textnormal{span}_\mathcal{R}(V_i)=\mathcal{S}_i$.  We will do so by showing that $\textnormal{span}_\mathcal{R}(V_i)$ contains every member of the Schur basis of $\mathcal{S}_i$.
\end{remark}

\begin{lemma}
\label{lemReducingHViaNewtonGirard}
Let $t\geq n-k+1$.  Then for some $d_{t,m}\in\mathcal{S}_{t-m}$ (where $n-k+1\leq m\leq n$) we have
$$
h_t = \sum_{m=n-k+1}^n d_{t,m} h_m
$$
\end{lemma}
\begin{proof}
We will induct on $t$.

Base case $n-k+1\leq t\leq n$: We can just take $d_{t,m}=\delta_{t,m}$ where $\delta$ is the Kronecker delta.

Induction step $t>n$: Suppose that Lemma~\ref{lemReducingHViaNewtonGirard} holds for all smaller $t\geq n-k+1$. From (\ref{equNewtonGirardhe}), we have
\begin{align*}
h_t &= \sum_{j=1}^k (-1)^{j+1}e_j h_{t-j}
\\
&= \sum_{j=1}^k (-1)^{j+1}e_j \sum_{m=n-k+1}^n d_{t-j,m} h_m
\\
&= \sum_{m=n-k+1}^n \left( \sum_{j=1}^k (-1)^{j+1}e_j d_{t-j,m} \right) h_m
\end{align*}
so we can take
$$
d_{t,m} = \sum_{j=1}^k (-1)^{j+1}e_j d_{t-j,m}
$$
which is in $\mathcal{S}_{j+t-j-m}=\mathcal{S}_{t-m}$ as desired.
\end{proof}

\begin{lemma}
\label{lemMixedBasisSHContainsSlambda}
Let $i\geq 0$ and $\lambda\in P_k(i)$.  Then $s_\lambda\in\textnormal{span}_\mathcal{R} V_i$.
\end{lemma}

\begin{proof}
We will induct on $i$.

Base case $i\leq n-k$: We have $\lambda_1\leq i\leq n-k$ so $s_\lambda=s_\lambda h_\emptyset \in V_i$.

Induction step $i>n-k$:  Suppose that Lemma~\ref{lemMixedBasisSHContainsSlambda} holds for all smaller $i$.  Let $\lambda\in P_k(i)$.  If $\lambda_1 \leq n-k$ then $s_\lambda=s_\lambda h_\emptyset \in V_i$, so let $\lambda_1\geq n-k+1$.

\begin{claim*}
In $\mathcal{S}_i$, for some $c_t\in \mathcal{S}_{i-t}$, we have
$$
s_\lambda = \sum_{t=\lambda_1}^i c_t h_t.
$$
\end{claim*}
\begin{proof}[Proof of the claim]
The Jacobi-Trudi identity (\ref{equJacobiTrudi}) states that
\begin{align}
\label{equExpandedJacobiTrudi}
s_\lambda &= \sum_{\sigma\in \mathfrak{S}_k} \text{sgn}(\sigma) h_{\lambda_1+\sigma(1)-1}\cdots h_{\lambda_k+\sigma(k)-k}
\end{align}
where $\mathfrak{S}_k$ is the symmetric group on $\{1,\dots,k\}$.  Since
$$
\deg(h_{\lambda_1+\sigma(1)-1}\cdots h_{\lambda_k+\sigma(k)-k})
= \sum_{i=1}^{k} \lambda_i + \sum_{j=1}^k j - \sum_{j=1}^k j
=
|\lambda|
=
i
$$
we see that each term in the sum (\ref{equExpandedJacobiTrudi}) is in $\mathcal{S}_i$.  Since $\lambda_1 \leq \lambda_1+\sigma(1)-1$ we can just factor $h_{\lambda_1+\sigma(1)-1}$ out of each term and recollect to obtain the desired expansion (there will be no $h_i$'s with $i>|\lambda|$ since then the corresponding term in (\ref{equExpandedJacobiTrudi}) would not be in $\mathcal{S}_i$).
\end{proof}

Now we can write that
$$
s_\lambda = \sum_{t=\lambda_1}^i c_t h_t
$$
where $c_t\in \mathcal{S}_{i-t}$.  Applying Lemma~\ref{lemReducingHViaNewtonGirard} to each $t$ in the sum (since $t\geq\lambda_1\geq n-k+1$) gives
\begin{align*}
s_\lambda &= \sum_{t=\lambda_1}^i c_t h_t
\\
&= \sum_{t=\lambda_1}^i c_t \sum_{m=n-k+1}^n d_{t,m} h_m
\\
&= \sum_{m=n-k+1}^n \left( \sum_{t=\lambda_1}^i c_t d_{t,m} \right) h_m.
\end{align*}
where $d_{t,m}\in\mathcal{S}_{t-m}$.  Note that $c_t d_{t,m}\in\mathcal{S}_{i-t+t-m}=\mathcal{S}_{i-m}$.  Since $i-m<i$, by our induction hypothesis, we can expand $c_t d_{t,m}$ as an $\mathcal{R}$-linear combination of $V_{i-m}$.  Then since multiplication by $h_m$, $n-k+1\leq m\leq n$, converts a member of $V_{i-m}$ into a member of $V_i$, we see that we have expanded $s_\lambda$ as an $\mathcal{R}$-linear combination of elements of $V_i$.  This concludes the induction step.
\end{proof}

\begin{proof}[Proof of Theorem~\ref{thmMixedBasisSlambdaHmu}]
Recall Remark~\ref{rmkStrategyForMixedBasisSlambdaHmu}.  In Lemma~\ref{lemMixedBasisSHContainsSlambda} we showed that if $\lambda\in P_k(i)$ then $s_\lambda\in \textnormal{span}_\mathcal{R}(V_i)$.  It follows that $\textnormal{span}_\mathcal{R}(V_i)=\mathcal{S}_i$, so $V_i$ is an $\mathcal{R}$-basis for $\mathcal{S}_i$ and the result follows.
\end{proof}

Now we will show that the $s_\lambda$'s of \ref{thmMixedBasisSlambdaHmu} may be replaced by $m_\lambda$'s or by $e_{\lambda'}$'s.

\begin{corollary}
\label{corMixedBasisMlambdaHmu}
The set
$$
\{m_\lambda h_\mu \,|\, \lambda\in P_{k,n-k}\text{, } \mu\in Q_{n-k+1,n}\}
$$
is an $\mathcal{R}$-basis for $\mathcal{S}$.
\end{corollary}
\begin{proof}
This is identical to Corollary~\ref{corMixedBasisSlambdaPmu} except with Theorem~\ref{thmMixedBasisSlambdaHmu} instead of Proposition~\ref{propMixedBasisMlambdaPmu}, with $s_\nu$ and $m_\nu$ switched, and with $h_\mu$ instead of $p_\mu$.
\end{proof}

\begin{corollary}
\label{corMixedBasisElambdaHmu}
The set
$$
\{e_{\lambda'} h_\mu \,|\, \lambda\in P_{k,n-k}\text{, } \mu\in Q_{n-k+1,n}\}
$$
is an $\mathcal{R}$-basis for $\mathcal{S}$.
\end{corollary}
\begin{proof}
This is identical to Corollary~\ref{corMixedBasisSlambdaPmu} except with Theorem~\ref{thmMixedBasisSlambdaHmu} instead of Proposition~\ref{propMixedBasisMlambdaPmu}, with $s_\nu$ instead of $m_\nu$, with $e_{\nu'}$ instead of $s_\nu$, with Lemma~\ref{lemEinSTriangularityPkn-k} instead of Lemma~\ref{lemSinMTriangularityPkn-k}, and with $h_\mu$ instead of $p_\mu$.
\end{proof}

\begin{remark}
Remark~\ref{rmkHalfofThmMain}, Theorem~\ref{thmMixedBasisSlambdaHmu}, Corollary~\ref{corMixedBasisMlambdaHmu}, and Corollary~\ref{corMixedBasisElambdaHmu} prove Theorem~\ref{thmMain}.
\end{remark}

\section{Quotients with Complete Homogeneous Symmetric Polynomials}

For this section let $\mathcal{R}$ be any commutative ring.  Recall that in the definition of $\mathcal{I}$ we had $a_i\in\mathcal{S}$ with $\deg(a_{n-k+i})<n-k+i$ for $i=1,\dots,k$.  Set $b_{n-k+i}=a_i$ (this is just a reindexing). In particular, $\deg(b_{n-k+i})<n-k+i$ and
$$
\mathcal{I} = \langle h_{n-k+1}-b_{n-k+1},\dots,h_n-b_n \rangle.
$$

First, we create slightly more general versions of the bases from Theorem~\ref{thmMixedBasisSlambdaHmu},
Corollary~\ref{corMixedBasisMlambdaHmu}, and Corollary~\ref{corMixedBasisElambdaHmu}.

\begin{corollary}
\label{corMixedBasesSMElambdaHmuAi}
The sets
\begin{align*}
\left\{\left.s_\lambda \prod_{j=1}^{\ell(\mu)} (h_{\mu_j}-b_{\mu_j}) \,\right|\, \lambda\in P_{k,n-k},\ \mu\in Q_{n-k+1,n}\right\}&,
\\
\left\{\left.m_\lambda \prod_{j=1}^{\ell(\mu)} (h_{\mu_j}-b_{\mu_j}) \,\right|\, \lambda\in P_{k,n-k},\ \mu\in Q_{n-k+1,n}\right\}&,
\\
\text{and }\left\{\left.e_{\lambda'} \prod_{j=1}^{\ell(\mu)} (h_{\mu_j}-b_{\mu_j}) \,\right|\, \lambda\in P_{k,n-k},\ \mu\in Q_{n-k+1,n}\right\}&
\end{align*}
are $\mathcal{R}$-bases for $\mathcal{S}$.
\end{corollary}
\begin{proof}
This is identical to Lemma~\ref{corMixedBasesSMElambdaPmuAi} except with Theorem~\ref{thmMixedBasisSlambdaHmu} instead of Corollary~\ref{corMixedBasisSlambdaPmu}, with Corollary~\ref{corMixedBasisMlambdaHmu} instead of Proposition~\ref{propMixedBasisMlambdaPmu}, and with Corollary~\ref{corMixedBasisElambdaHmu} instead of Corollary~\ref{corMixedBasisElambdaPmu}.
\end{proof}

\begin{corollary}
\label{corSMEGrinQuoBasis}
The sets 
\begin{align*}
\{s_\lambda\,&|\,\lambda\in P_{k,n-k}\},
\\
\{m_\lambda\,&|\,\lambda\in P_{k,n-k}\},
\\
\text{and } \{e_{\lambda'}\,&|\,\lambda\in P_{k,n-k}\}
\end{align*}
are $\mathcal{R}$-bases for $\mathcal{S}/\mathcal{I}$.
\end{corollary}
\begin{proof}
This is identical to Corollary~\ref{corSMEWeinQuoBasis} except with Corollary~\ref{corMixedBasesSMElambdaHmuAi} instead of Corollary~\ref{corMixedBasesSMElambdaPmuAi}.
\end{proof}

\begin{remark}
Corollary~\ref{corSMEWeinQuoBasis} and Corollary~\ref{corSMEGrinQuoBasis} prove Corollary~\ref{corMainSME}.
\end{remark}

\section{Conjectures}

We have thus far shown that all of Grinberg's bases for $\mathcal{S}/\mathcal{I}$ are also bases for $\mathcal{S}/\mathcal{J}$ except for the $\{h_\lambda\}_{\lambda\in P_{k,n-k}}$ basis. Here we conjecture two analogs:

\begin{conjecture}
\label{conjMixedBasisHlambdaPmu}
The set
$$
\{ h_\lambda p_\mu\,|\,\lambda\in P_{k,n-k},\ \mu\in Q_{n-k+1,n} \}
$$
is an $\mathcal{R}$-basis for $\mathcal{S}$, and the set
$$
\{ h_\lambda\,|\,\lambda\in P_{k,n-k} \}
$$
is an $\mathcal{R}$-basis for $\mathcal{S}/\mathcal{J}$.
\end{conjecture}

\begin{conjecture}
\label{conjMixedBasisHlambda'Pmu}
The set
$$
\{ h_{\lambda'} p_\mu\,|\,\lambda\in P_{k,n-k},\ \mu\in Q_{n-k+1,n} \}
$$
is an $\mathcal{R}$-basis for $\mathcal{S}$, and the set
$$
\{ h_{\lambda'}\,|\,\lambda\in P_{k,n-k} \}
$$
is an $\mathcal{R}$-basis for $\mathcal{S}/\mathcal{J}$ (note the conjugates).
\end{conjecture}

\section{Acknowledgements}

The author would like to thank Guangyi Yue for her guidance and mentorship and Darij Grinberg for proposing this project and making helpful suggestions.  The author would also like to thank Yongyi Chen and Tanya Khovanova for reviewing this paper and the MIT PRIMES program, particularly Pavel Etingof and Slava Gerovitch, for making this research possible.


\begin{thebibliography}{999999}

\bibitem[Aig07]{aigner}
Martin Aigner. \textit{A Course in Enumeration}. Springer Science \& Business Media, 2007.

\bibitem[ER91]{egecioglu} {\"O}mer E{\~g}ecio{\~g}lu and Jeffrey B Remmel. \textit{Brick tabloids and the connection matrices between bases of symmetric functions}. Discrete Applied Mathematics 34.1-3 (1991), pp. 107-120.\\ \texttt{\url{https://sites.cs.ucsb.edu/~omer/DOWNLOADABLE/bricktabloids91.pdf}}

\bibitem[GR19]{expansions} Darij Grinberg and Victor Reiner. \textit{Hopf algebras in combinatorics}. \href{https://arxiv.org/abs/1409.8356v5}{\texttt{arXiv:1409.8356v5}}.

\bibitem[Gri19]{basisquot} Darij Grinberg. \textit{A basis for a quotient of symmetric polynomials}. \href{https://arxiv.org/abs/1910.00207}{\texttt{arXiv:1910.00207v1}}.

\bibitem[Mac98]{macdonald} Ian Macdonald. \textit{Symmetric functions and Hall polynomials.} 2nd ed. Oxford university press, 1998. \\ \texttt{\url{https://pdfs.semanticscholar.org/0613/1de77b4268cc9d4334a661846c42873cb8e4.pdf}}.

\bibitem[Pos05]{postnikov} Alexander Postnikov. \textit{Affine approach to quantum Schubert calculus}. Duke Mathematical Journal 128.3 (2005), pp. 473-509. \texttt{\url{https://math.mit.edu/~apost/papers/affine_approach.pdf}}.

\bibitem[Sta11]{enum1} Richard P Stanley. \textit{Enumerative Combinatorics}. 2nd ed. Vol 1. Cambridge Studies in Advanced Mathematics, 2011. \\ \texttt{\url{http://www-math.mit.edu/~rstan/ec/ec1.pdf}}.

\bibitem[Sta99]{enum2} Richard P Stanley. \textit{Enumerative Combinatorics}. Vol 2. Cambridge Studies in Advanced Mathematics, 1999.

\bibitem[Vas69]{rankandspanimplybasis}
Wolmer V Vasconcelos. \textit{On finitely generated flat modules}.  Transactions of the American Mathematical Society 138 (1969), pp. 505-512.

\end{thebibliography}
\end{document}